\documentclass[11pt]{amsart}
\usepackage{amsmath,amsthm,amsfonts,amssymb,times}

\usepackage{amssymb}
\usepackage{amscd}
\usepackage{amsfonts}
\usepackage{version}

\usepackage{graphicx}

\newtheorem{theorem}{Theorem}[section]

\newtheorem{proposition}[theorem]{Proposition}

\theoremstyle{definition}

\newtheorem{definition}{Definition}[section]

\renewcommand{\leq}{\leqslant}
\renewcommand{\geq}{\geqslant}

\newcommand\codim{\operatorname{codim}}

\def\F{\mathbb{F}}

\def\C{\mathbb{C}}

\def\E{\mathbb{E}}

\def\N{\mathbb{N}}

\def\eps{\varepsilon}

\numberwithin{equation}{section}

\begin{document}

\title{Fourier uniformity on subspaces}

\author{Ben Green}
\address{Mathematical Institute\\
University of Oxford\\
Radcliffe Observatory Quarter\\
Woodstock Road\\
Oxford OX2 6GG\\
United Kingdom}
\email{ben.green@maths.ox.ac.uk}

\author{Tom Sanders}
\address{The Mathematical Institute, Radcliffe Observatory Quarter, Woodstock Road, Oxford OX2 6GG}
\email{tom.sanders@maths.ox.ac.uk}

\begin{abstract}
Let $\F$ be a fixed finite field, and let $A \subset \F^n$. It is a well-known fact that there is a subspace $V \leq \F^n$, $\codim V \ll_{\delta} 1$, and an $x$, such that $A$ is $\delta$-uniform when restricted to $x + V$ (that is, all non-trivial Fourier coefficients of $A$ restricted to $x + V$ have magnitude at most $\delta$). We show that if $\F = \F_2$ then it is possible to take $x = 0$; that is, $A$ is $\delta$-uniform on a subspace $V \leq \F^n$. We give an example to show that this is not necessarily possible when $\F = \F_3$.
\end{abstract}

\maketitle

\section{Introduction}

Let $\F$ be a fixed finite field of prime order $p$, and consider the vector space $\F^n$. We identify $\F^n$ with its own dual via the dot product, and in this way define the Fourier transform of a function $f : \F^n \rightarrow \C$ by
\[ \hat{f}(r) := \E_{x \in \F^n} f(x) e_p(-r \cdot x),\] where $e_p(t) := e^{2\pi i t/p}$ and $r$ takes values in $\F^n$.  When $\F = \F_2$, we have $e_p(t) = (-1)^t$. For any non-empty set $S \subset G$ we write $\mu_S$ for the uniform probability measure induced on $S$, that is the probability measure assigning mass $|S|^{-1}$ to each $s \in S$. 

\begin{definition}
Suppose that $A \subset \F^n$ is a set and that $V \leq \F^n$ is a subspace. Let $x \in \F^n$. Then we say that $A$ is $\eps$-uniform on the coset $x + V$ if
\[ \sup_{r \notin V^{\perp}}| (1_A \mu_{x + V})^{\wedge}(r)| \leq \eps.\]
\end{definition}

The following fact, proven by a ``density increment argument'' is well-known in the additive combinatorics literature and is implicit, for example, in the work of Meshulam \cite{mes::}.

\begin{theorem}
Suppose that $A \subset \F^n$ is a set. Then there is a subspace $V \leq \F^n$, $\codim V \ll \eps^{-1} $, and an $x \in \F^n$ such that $A$ is $\eps$-uniform on the coset $x + V$. 
\end{theorem}

Our aim in this note is to prove the following.

\begin{theorem}\label{mainthm}
Suppose that $A \subset \F_2^n$ is a set. Then there is a subspace $V \leq \F^n$, $\codim V \ll_{\eps} 1$, such that $A$ is $\eps$-uniform on $V$.
\end{theorem}

\emph{Remarks.} Sadly, the implied constant in $\eps$ is atrocious, being a tower of towers of height $O(\eps^{-1})$. It would be interesting to get a better bound. Note that it is quite permissible for $A \cap V$ to be empty, and indeed this is generally unavoidable, as the example $A = \{x : x_1 = 1\}$ shows. 

In Section \ref{f3-example} we give a simple example to show that this statement is not true when $\F_2$ is replaced by $\F_3$. 

\section{An example over $\F_3$}\label{f3-example}

In this section we give a simple example to show that the analogue of Theorem \ref{mainthm} is false over $\F_3$ (similar examples may be constructed over other prime fields). The example comes from the literature on Rado's theorem over finite fields, in particular from \cite{BDH}. Indeed, if Theorem \ref{mainthm} \emph{had} been true over $\F_3$ it would have implied that every homogeneous equation in three or more variables is partition regular in $\F_3^{\N}$, a result which is known to be false.

\begin{theorem}\label{thm.main2}
There is a set $A \subset \F_3^n$ such that for any subspace $V \leq \F_3^n$ of positive dimension we have
\[ \sup_{r \notin V^{\perp}} |(1_A d\mu_V)^{\wedge}(r)| \geq \frac{\sqrt{3}}{6}.\]
\end{theorem}
\begin{proof}
Take $A = \{ x \in \F_3^n : \mbox{there exists $i$ such that } \; x_1 = \dots = x_i = 0, x_{i + 1} = 1\}$.
Let $V \leq \F_3^n$ be a subspace of positive dimension, and let $j \in [n]$ be minimal such that $v_j \neq 0$ for at least one $v \in V$. Of course, we then have $x_1 = \dots = x_{j-1} = 0$ for all $x \in V$. It follows that \begin{equation}\label{inclu-1}\{ x \in V: x_j = 1\} \subset A,\end{equation} whilst \begin{equation}\label{inclu-2}\{x \in V: x_j = 2\} \cap A = \emptyset.\end{equation} Take $r \in \F_3^n$ to have $r_1 = \dots = r_{j-1},r_{j+1},\dots,r_n = 0$ and $r_j = 1$. Then $r \notin V^{\perp}$, since $r \cdot v \neq 0$. Furthermore, a short computation using \eqref{inclu-1} and \eqref{inclu-2} gives
\[ \mbox{Im} \big((1_A d\mu_V)^{\wedge}(r)\big) = -\frac{\sqrt{3}}{2} \mu_V \big( \{ x \in V: x_j = 1\}\big) = - \frac{\sqrt{3}}{6},\]
the second equality being a consequence of the fact that the map $V \rightarrow \F_3$ given by $x \mapsto x_j$ is linear and nontrivial.
\end{proof}

\section{A Ramsey result for almost colourings}

By a $(1 - \delta)$-almost $r$-colouring of a set $X$, we mean a map $c : \tilde X \rightarrow [r]$ where $|X \setminus \tilde X| \leq \delta |X|$.

\begin{proposition}\label{finite-union}
Let $r,d$ be integers. Then there is an $\eta(r,d) > 0$ such that the following is true. If $n \geq n_0(r,d)$ is sufficiently large, $\eta \in [0,\eta(r,d)]$, and if we have a $(1 - \eta)$-almost $r$-colouring of $\F_2^n$, then there are linearly independent $x_1,\dots, x_d$ such that all of the sums $\sum_{i \in I} x_i$, $I \subset [d]$, $I \neq \emptyset$, are the same colour.
\end{proposition}
\begin{proof}
In the case $\eta = 0$ (that is, genuine $r$-colourings rather than almost- colourings) this follows quickly from a well-known theorem of Graham and Rothschild \cite[Corollary 1]{gr2} (for a short proof see \cite{nesetril-rodl}). Indeed, our colouring of $\F_2^n$ induces a colouring of the power set $\mathcal{P}([n])$ via the usual identification of these two sets by characteristic functions. In the power set $\mathcal{P}([n])$, the theorem of Graham and Rothschild guarantees that if $n$ is sufficiently large then there are disjoint subsets $S_1,\dots, S_d \subset [n]$ such that every nontrivial union $\bigcup_{i \in I} S_i$, $I \subset [d]$, $I \neq \emptyset$, is the same colour. These sets pull back under the identification to give $x_1,\dots, x_d$ with the claimed property.

We may deduce the stronger result claimed (that is, with $\eta > 0$) by a simple averaging argument. Let $m = m(r,d)$ be a value of $n$ for which the result is true with $\eta = 0$. Now take $\eta(r,d):= 2^{-m-1}$, and suppose $\eta \in [0,\eta(r,d)]$ and we have a $(1 - \eta)$-almost colouring of $\F_2^n$. If $n \geq m+3$, this induces a $(1 - 2\eta(r,d))$-almost colouring of $\F_2^n \setminus \{0\}$. Now $\F_2^n \setminus \{0\}$ is uniformly covered by sets $V \setminus \{0\}$, where $V$ ranges over all $m$-dimensional subspaces of $\F_2^n$. Therefore, by the pigeonhole principle, there is some $V$ for which we get an induced $(1 - 2\eta(r,d))$-almost colouring of $V \setminus \{0\}$. However, since $1 - 2\eta(r,d) \geq 1- 2^{-m} > 1 - \frac{1}{|V \setminus \{0\}|}$, this is in fact a full colouring of $V \setminus \{0\}$. The result follows by the choice of $m$.
\end{proof}

\section{Proof of the main theorem}

Suppose that $A \subset \F_2^n$ is a set and we are given a parameter $\eps \in (0,1]$. Choose integers $d, r$ with $2^d \sim r \sim \frac{1}{\eps}$, and let $\eta = \eta(r,d)$ be the parameter whose existence is guaranteed by Proposition \ref{finite-union}. By the ``arithmetic regularity lemma'' in this context \cite{gre::02}, there is\footnote{Strictly speaking, the argument as presented in \cite{gre::02} does not guarantee a \emph{lower} bound on $\codim W$. However, this may very easily be arranged with a trivial modifcation of the proof, for example by foliating $\F_2^n$ into cosets of some arbitrary subspace of codimension $n_0(r,d)$ and then running the energy increment argument as in that paper.} some subspace $W \leq \F_2^n$, $n_0(r,d) \leq \codim W \ll_{\eps} 1$, such that 
\begin{equation}\label{reg} \sup_{r \notin W^{\perp}} |(1_A \mu_{x + W})^{\wedge}(r)| \leq \eps\end{equation} for a proportion at least $1 - \eta$ of all $x \in \F_2^n$. For notational simplicity, put $X:=\F_2^m$ and change basis so that $W = \{0_X\} \times \F_2^{n - m}$. Let $\tilde X \subset X$ be the set of all $x \in \F_2^m$ for which \eqref{reg} holds. Thus $|\tilde X| \geq (1 - \eta)|X|$. Define an $(r+1)$-colouring $c : \tilde X \rightarrow \{0,1,\dots, r\}$ (and hence a $(1 - \eta)$-almost $(r+1)$-colouring of $X$) by defining $c(x) := \lfloor r \E_{x + V} 1_A  \rfloor$. That is, $c(x) = j$ if the density of $A$ on $x + V$ lies in the range $[\frac{j}{r}, \frac{j+1}{r})$. By Proposition \ref{finite-union}, we may find linearly independent $x_1,\dots, x_d \in \F_2^m$ and $j \in \{0,1,\dots, r\}$ such that $c(\sum_{i \in I} x_i) = j$ for all $I \subset [d]$, $I \neq \emptyset$. 

Set $V := W + \langle x_1,\dots, x_d\rangle$. We claim that 
\begin{equation}\label{to-check} \sup_{r \notin V^{\perp}} |(1_A \mu_V)^{\wedge}(r)| = O(\eps),\end{equation} which (after redefining $\eps$ to $\eps/C$) implies our main theorem.
Suppose first that $r \notin W^{\perp}$. Note that 
\begin{equation}\label{tbn} \mu_V = 2^{-d} \sum_{I \subset [d]} \mu_{W + \sum_{i \in I} x_i}.\end{equation} If $I \neq \emptyset$ we have $\sum_{i \in I} x_i \in \tilde X$, and so in this case
\[ |(1_A \mu_{\sum_{i \in I} x_i + W})^{\wedge}(r)| \leq \eps\] by \eqref{reg}. Summing over all $I \neq \emptyset$, and handling the case $I = \emptyset$ trivially, we have
\[ |(1_A \mu_V)^{\wedge}(r)| \leq 2^{-d} + \eps = O(\eps),\] as desired.
Now suppose that $r \in W^{\perp} \setminus V^{\perp}$. In this case
\[ (1_A \mu_{x + W})^{\wedge}(r) = (-1)^{r \cdot x} \E_{x + W} 1_A.\] Hence from \eqref{tbn} we have
\[ (1_A \mu_V)^{\wedge}(r) = 2^{-d} \sum_{I \subset [d]} (-1)^{r \cdot \sum_{i \in I} x_i} \E_{\sum_{i \in I} x_i + W} 1_A.\]
By construction,
\[ \frac{j}{r} \leq \E_{\sum_{i \in I} x_i + W} 1_A < \frac{j+1}{r}\] whenever $I \neq \emptyset$. It follows that 
\[ (1_A \mu_V)^{\wedge}(r) = \frac{j}{r} 2^{-d} \sum_{I \subset [d]} (-1)^{r \cdot \sum_{i \in I} x_i} + O(2^{-d} + \frac{1}{r}).\]
However
\[ \sum_{I \subset [d]} (-1)^{r \cdot \sum_{i \in I} x_i}  = \prod_{i = 1}^d (1 + (-1)^{r \cdot x_i}),\] and at least one of the factors here vanishes since $r \notin V^{\perp}$. Hence 
\[ (1_A \mu_V)^{\wedge}(r) = O(2^{-d} + \frac{1}{r}) = O(\eps),\] and the proof is complete.

\end{document}